\documentclass[12pt,twoside,reqno]{amsart}
\usepackage[colorlinks=true,citecolor=blue]{hyperref}
\usepackage{mathptmx, amsmath, amssymb, amsfonts, amsthm, mathptmx, enumerate, color,mathrsfs}
\usepackage{graphicx}

\usepackage{multirow}
\usepackage{epstopdf}
\usepackage{multicol}
\usepackage{epstopdf}
\usepackage{algorithm2e}

\newtheorem{theorem}{Theorem}[section]

\newtheorem{proposition}[theorem]{Proposition}

\newtheorem{condition}[theorem]{Condition}
\theoremstyle{definition}
\newtheorem{definition}[theorem]{Definition}

\newtheorem{problem}[theorem]{Problem}
\newtheorem{example}[theorem]{Example}

\newtheorem{remark}[theorem]{Remark}
\numberwithin{equation}{section}

\begin{document}
\setcounter{page}{1}

\vspace*{2.0cm}
\title[feasibility-seeking with sets
that are unions of convex sets]
{An iterative process for the feasibility-seeking problem with sets
that are unions of convex sets}
\author[Y. Censor, A.J. Zaslavski]{ Yair Censor$^{1,*}$, Alexander J. Zaslavski$^2$}
\maketitle
\vspace*{-0.6cm}

\begin{center}
{\footnotesize

$^1$Department of Mathematics, University of Haifa, Mt. Carmel, Haifa 3498838, Israel\\
$^2$Department of Mathematics, The Technion-Israel Institute of Technology, Haifa 3200003, Israel\\
\bigskip

\textbf{This work is dedicated to the memory of Professor Jon Borwein}

}\end{center}

\vskip 4mm {\footnotesize \noindent {\bf Abstract.}
In this paper we deal with the feasibility-seeking problem for unions of
convex sets (UCS) sets and propose an iterative process for its solution.
Renewed interest in this problem stems from the fact that it was recently
discovered to serve as a modeling approach in fields of applications and from
the ongoing recent research efforts to handle non-convexity in feasibility-seeking.

 \noindent {\bf Keywords.}
 Feasibility-seeking; non-convex sets; unions of convex
sets; floorplanning; projections onto sets expandable in convex sets.

 \noindent {\bf 2020 Mathematics Subject Classification.}
90C30, 65K10. }

\renewcommand{\thefootnote}{}
\footnotetext{ $^*$Corresponding author.
\par
E-mail address: yair@math.haifa.ac.il (Y. Censor), ajzasl@technion.ac.il (A.J. Zaslavski).
\par
Received May, 17, 2024; Accepted April, 1, 2025.

\rightline {\tiny   \copyright  2022 Communications in Optimization Theory}}

\section{Introduction\label{sec:intro}}

The feasibility problem is to find a point in the nonempty intersection of a
finite family of sets. It is a modeling framework for many real-world problems
and problems in physics and mathematics. The literature about it in the convex
case, when all sets are convex, is huge and diverse, see, e.g., \cite{bb96},
\cite{Cegielski-book} and \cite{censor-cegielski-review}. The non-convex
situation is more difficult to address but has witnessed many works in recent
years, see, e.g., \cite{hesse-luke}, \cite{pong-li-2016}, \cite{dizon-2022},
\cite{bauschke-2014} and \cite{corvellec-2004}, to name but a few.

A particular kind of non-convexity that occurs when the sets for the
feasibility problem are unions of convex sets was studied by Chrétien and
Bondon \cite{chertien-bondon-1996}, \cite{cheritien-bondon-2020}. This was
recently discovered to serve as a modeling approach in the application field
of floorplanning with I/O (input/output) assignment \cite{Shan-Yu-2023},
bringing up again the topic of investigating iterative processes for it.
Recent work \cite{bauschke-noll-2014} discusses the Douglas--Rachford
algorithm to solve the feasibility problem for two closed sets that are finite
unions of convex sets.

Related, although indirectly, to this subject are the recent papers on
algorithms that are based on unions of nonexpansive maps, \cite{tam-2018},
\cite{dao-2019} and \cite{zaslavski-2023}.

In this paper we deal with the feasibility-seeking problem for unions of
convex sets (UCS) sets and propose an iterative process that is different from
earlier proposed ones. Further work is needed to assess the practicality of
this algorithm as well as that of the earlier proposed ones mentioned above.

We use the term \textquotedblleft Feasibility-Seeking
Problem\textquotedblright\ and not just \textquotedblleft Feasibility
Problem\textquotedblright\ because the latter sometimes refers to the problem
of deciding whether or not an intersection of sets is or is not feasible,
i.e., nonempty, see, e.g., \cite{Firouzeh-2022}. The term we use here more
accurately describes that one is interested in finding a point rather than
deciding feasibility or manipulating the problem to reach feasibility.

\section{Notions and notations\label{sec:preliminaries}}

Let $(\mathcal{H},\langle\cdot,\cdot\rangle)$ be a real Hilbert space equipped
with an inner product $\langle\cdot,\cdot\rangle$ which induces the Euclidean
norm $\Vert x\Vert=\langle x,x\rangle^{1/2}$, $x\in\mathcal{H}$. For each
$x\in\mathcal{H}$ and each $r>0$ the closed ball with radius $r$ centered at
$x$ is $B(x,r):=\{y\in\mathcal{H}\mid\;\Vert x-y\Vert\leq r\}.$

For each $x\in\mathcal{H}$ and each nonempty set $\Omega\subset\mathcal{H}$
the distance between $x$ and the set $\Omega$ is $dist(x,\Omega):=\inf\{\Vert
x-y\Vert\mid\;x\in\Omega\}.$

Let $\Omega\subset\mathcal{H}$ be a nonempty, convex and closed set. Then for
each $x\in\mathcal{H}$ there exists a unique point $P_{\Omega}(x)\in\Omega$,
called the \textbf{orthogonal projection of }$x$\textbf{ onto }$\Omega$, such
that
\begin{equation}
\Vert x-P_{\Omega}(x)\Vert=dist(x,\Omega). \label{eq:2.1}%
\end{equation}

\section{The feasibility-seeking problem for sets that are unions of convex
sets\label{sec:the-problem}}

Sets that are unions of convex sets are defined as follows.

\begin{definition}
\label{def:UCS}\textbf{(union of convex sets set). }Let $m_{\Omega}$ be a
natural number and let%
\begin{equation}
\Omega_{s}\subset\mathcal{H},\;s=1,2,\dots,m_{\Omega},
\end{equation}
be nonempty closed convex sets that are pairwise disjoint, i.e.,%
\begin{equation}
\Omega_{s}\cap\Omega_{t}=\emptyset,\text{ \ for all }s,t\in\{1,2,\dots
,m_{\Omega}\},\text{ }s\not =t. \label{eq:pairwise-disjoint}%
\end{equation}
The union of these sets%
\begin{equation}
\Omega:=\cup_{s=1}^{m_{\Omega}}\Omega_{s} \label{eq:USC-set}%
\end{equation}
will be called a \textquotedblleft\textbf{union of convex sets (UCS)
set}\textquotedblright.
\end{definition}

Such sets were called, under different conditions, \textquotedblleft
expandable in convex sets\textquotedblright\ in \cite[Definition
1]{chertien-bondon-1996} and in \cite[Definition 1]{cheritien-bondon-2020}. In
\cite[Subsection 6.1]{dao-2019} such sets were considered under the name
\textquotedblleft union convex sets\textquotedblright. We define, the possibly
non-unique, orthogonal projections onto a UCS set as follows.

\begin{definition}
\label{def:proj-on-ucs}Let $\Omega$ be a UCS set, $\Omega:=\cup_{s=1}%
^{m_{\Omega}}\Omega_{s}$. For each $x\in\mathcal{H}$ the standard definition
of projection onto $\Omega$ is%
\begin{equation}
P_{\Omega}(x):=\{y\in\Omega\mid\;\Vert x-y\Vert\leq\Vert z-x\Vert,\text{ \ for
all }z\in\Omega\}.
\end{equation}

\end{definition}

\begin{remark}
\label{remark:proj}Let $T\subset\{1,2,\dots,m_{\Omega}\}$ be the set of
indices of the sets $\Omega_{t}$ whose distances to $x$ are smallest, compared
to the distances of $x$ to the other sets $\Omega_{s}$. Then $P_{\Omega}(x)$
is the, obviously nonempty, set%
\begin{equation}
P_{\Omega}(x)=\{P_{\Omega_{t}}(x),\;t\in T\mid\Vert x-P_{\Omega_{t}}%
(x)\Vert\leq\Vert x-P_{\Omega_{s}}(x)\Vert,\;s=1,2,\dots,m_{\Omega}\}.
\end{equation}

\end{remark}

The feasibility-seeking problem for UCS sets is defined next.

\begin{problem}
\label{prob:problem}(\textbf{The} \textbf{feasibility-seeking problem for UCS
sets). }Let $m$ be a natural number and assume that, for each $i\in
\{1,2,\dots,m\}$, the set $C_{i}$ is a UCS set%
\begin{equation}
C_{i}:=\cup_{j=1}^{m_{i}}C_{i,j},\label{eq:USC-set copy(2)}%
\end{equation}
so that the integer $m_{i}$ is the number of sets whose union constitutes the
set $C_{i}.$ The feasibility-seeking problem for UCS sets is to find a point
$x^{\ast}$ in the intersection%
\begin{equation}
x^{\ast}\in C:=\cap_{i=1}^{m}C_{i}.\label{eq: (8)}%
\end{equation}

\end{problem}

Throughout this work we assume that the feasibility-seeking problem for UCS
sets is feasible, i.e., that the intersection is nonempty $C\not =\emptyset.$

\section{The iterative process\label{sec:iter-process}}

During the iterative process we will perform projections onto the sets $C_{i}%
$, which are themselves unions of sets. Such projections might not be unique.
For the definition of our algorithm\footnote{As common, we use the term
algorithm for the iterative process studied here although no termination
criteria, which are by definition necessary in an algorithm, are present and
only the asymptotic behavior is studied.} and for its convergence analysis we
use a condition that guarantees the uniqueness of the projections onto the UCS
sets involved in the problem. To this end we append Problem \ref{prob:problem}
with an additional UCS set which is a copy of $C_{1}.$ This clearly does not
change at all the problem. Thus, we define%
\begin{equation}
C_{m+1}=C_{1},\text{ \ }m_{m+1}=m_{1,}\text{ \ }C_{m+1,j}=C_{1,j}\text{ \ for
all }j\in\{1,2,\dots,m_{1}\}. \label{eq:additional-set}%
\end{equation}

\begin{condition}
\label{cond:condition}For each $i\in\{1,2,\dots,m\}$ and each $j\in
\{1,2,\dots,m_{i}\}$ there exists a unique integer%
\begin{equation}
\theta(i,j)\in\{1,2,\dots,m_{i+1}\}
\end{equation}
such that for every $x\in C_{i,j}$%
\begin{equation}
P_{C_{i+1}}(x)=P_{C_{i+1,\theta(i,j)}}(x), \label{eq:3.4}%
\end{equation}
and%
\begin{equation}
dist(x,C_{i+1,\theta(i,j)})<dist(x,C_{i+1,\ell}),\;\ell\in\{1,2,\dots
,m_{i+1}\}\setminus\{\theta(i,j)\}. \label{eq:3.5}%
\end{equation}

\end{condition}

Observe that the inequality (\ref{eq:3.5}) does not follow from the pairwise
disjointedness condition (\ref{eq:pairwise-disjoint}) in Definition
\ref{def:UCS} because the latter does not rule out that there could be two (or
more) sets in the union of sets to which they belong that will have equal
distances to a point $x.$ Condition \ref{cond:condition} plays a crucial role
in our study. Proposition \ref{prop:prop2.1}, presented in the sequel, shows
that it holds if for each $i\in\{1,2,\dots,m-1\}$ and each $x\in C_{i}$,
$P_{C_{i+1}}(x)$ is a singleton.

For all natural numbers $k$ define the cyclic \textquotedblleft control
sequence\textquotedblright\ $\{i(k)\}_{k=1}^{\infty}$ such that $i(k)\in
\{1,2,\dots,m\}$ for all $k$ and%
\begin{equation}
i(k):=(k-1)\operatorname{mod}m+1.
\end{equation}
This sequence serves to index the sets that are used by the algorithm which is now described.

\bigskip

\textbf{Algorithm 1. Projections onto Unions of Convex Sets (PUCS)
Algorithm}\\
(1) \textbf{Initialization}: Set $i=1$ and pick $C_{1}.$ For each
$r\in\{1,2,\dots,m_{1}\}$ pick an arbitrary initial point
\begin{equation}
y_{r}^{1}\in C_{1,r} \label{algeq:init}
\end{equation}
and define
\begin{equation}
\tau(r,1):=r. \label{algeq:parameter}
\end{equation}

(2) \textbf{First sweep}:

Set $k=1$ and as long as $k\leq m$ do:

Given an iteration vector $y_{r}^{k}$ calculate the next iteration vector for
it by
\begin{equation}
y_{r}^{k+1}=P_{C_{i(k+1)}}(y_{r}^{k}),
\end{equation}
and denote the index of the set in the UCS set $C_{i(k+1)}$ which is closest
to $y_{r}^{k}$ by $\tau(r,k+1)\in\{1,2,\dots,m_{i(k+1)}\},$ so that
\begin{equation}
y_{r}^{k+1}\in P_{C_{i(k+1),\tau(r,k+1)}}(y_{r}^{k}).
\end{equation}

After $m$ consecutive iterations of these iterative steps we have completed
the first sweep and we reach, for each $r\in\{1,2,\dots,m_{1}\},$
\begin{equation}
y_{r}^{m+1}=P_{C_{m+1}}(y_{r}^{m})=P_{C_{1}}(y_{r}^{m})\text{ \ and \ }
\tau(r,m+1).\text{ }
\end{equation}

Define
\begin{equation}
R:=\{r\mid1\leq r\leq m_{1}\text{ \ for which \ }\tau(r,m+1)\neq r\text{ \ has
occurred}\} \label{eq:the-set-R}
\end{equation}
and go to Step (3).

(3) \textbf{Iterative step}: For all $r\in\{1,2,\dots,m_{1}\}\setminus R$ set
$k=1$ and $y_{r}^{1}\leftarrow y_{r}^{m+1},$ obtained from Step (2), and do
for all $k\geq1$:

Given an iteration vector $y_{r}^{k}$ calculate the next iteration vector for
it by
\begin{equation}
y_{r}^{k+1}=P_{C_{i(k+1)}}(y_{r}^{k}),
\end{equation}
and denote the index of the set in the UCS set $C_{i(k+1)}$ which is closest
to $y_{r}^{k}$ by $\tau(r,k+1)\in\{1,2,\dots,m_{k+1}\},$ so that actually
\begin{equation}
y_{r}^{k+1}\in P_{C_{i(k+1),\tau(r,k+1)}}(y_{r}^{k}).
\end{equation}

Observe that (\ref{eq: (8)}) and Condition \ref{cond:condition} guarantee that
$R\neq\{1,2,\dots,m_{1}\}$ thus allowing the iterative step (3) in the
algorithm to proceed. We give now a verbal description that sheds some light
on the logic behind the algorithm. For the purpose of this discussion, let us
nickname the sets $C_{i}$ that constitute the feasibility-seeking problem as
\textquotedblleft the large sets\textquotedblright\ and nickname the sets
whose unions constitute large sets as \textquotedblleft inner
sets\textquotedblright.

The set $C$ must be an intersection of a family of inner sets that includes at
least one inner set from each large set $C_{i}.$ The multitude of inner sets
presents the algorithm with the, nontrivial, task of \textquotedblleft
discovering\textquotedblright\ those inner sets whose intersection is the set
$C.$ The algorithm is initialized by picking one large set, say $C_{1},$ and
picking a set of arbitrary points, each from one of its inner sets as in
(\ref{algeq:init}). This is done because it is not known, at the beginning,
which one of the inner sets of $C_{1}$ participates in the intersection $C.$
Integer parameters are set as in (\ref{algeq:parameter}).

From each of those initial points $y_{r}^{1}$ the algorithm preforms\ a full
sequential cycle (nicknamed \textquotedblleft a sweep\textquotedblright) of
successive projections onto the remaining large sets. For each projection the
parameter $\tau(r,k+1)$ (for every $r\in\{1,2,\dots,m_{1}\}$) indicates the
specific inner set that was projected on. The algorithm is parallel in the
sense that the latter activity, as well as others in it, can be performed
simultaneously on several processors.

At the end of this first sweep we have a set of \textquotedblleft
end-points\textquotedblright\ $y_{r}^{m+1}$ and the parameters $\tau(r,m+1)$
associated with each. The parameters tell the index of the inner set on which
the last projection of each sweep occurred. An end-point for which
$\tau(r,m+1)=r$ means that the orbit initiated by it returns to the inner set
$C_{1,r}$ from which it was initialized. These are the orbits that we wish to
follow, thus, the set $R$ is defined as in (\ref{eq:the-set-R}) and its
indices are deleted at the beginning of the iterative step (3). We need the
initial sweep to identify orbits on whose progress the algorithm
\textquotedblleft has no control\textquotedblright\ about which we can not say
where they are leading to.

The remaining orbits, for $r\in\{1,2,\dots,m_{1}\}\setminus R$ are being
retained and sequential cyclic successive projections are performed along
them. The convergence result in Theorem \ref{theorem:converge} shows that the
distance between each sequence $\{y_{r}^{k}\}_{k=1}^{\infty},$ generated by
the algorithm, and any large set converges to zero. In fact, the algorithm
solves, in parallel, some pruned convex feasibility-seeking problems via an
iterative sequential projection method.

The convergence of Algorithm 1 can now be established as
follows in the next theorem.

\begin{theorem}
\label{theorem:converge}Consider feasibility-seeking for UCS sets of Problem
\ref{prob:problem} with an additional UCS set as in (\ref{eq:additional-set}).
Assume that the problem is feasible and that Condition \ref{cond:condition}
holds. Let $r\in\{1,2,\dots,m_{1}\}$ be such that%
\begin{equation}
C\cap C_{1,r}\not =\emptyset\label{eq:22}%
\end{equation}
and let $\{y_{r}^{k}\}_{k=1}^{\infty}$ be a sequence generated by Algorithm 1. Then%
\begin{equation}
\lim_{k\rightarrow\infty}\Vert y_{r}^{k}-y_{r}^{k+1}\Vert=0\label{eq:lim-0}%
\end{equation}
and for each $i\in\{1,2,\dots,m\},$%
\begin{equation}
\lim_{k\rightarrow\infty}dist(y_{r}^{k},C_{i})=0.\label{eq:(24)}%
\end{equation}

\end{theorem}

\begin{proof}
There exists a%
\begin{equation}
z\in C\cap C_{1,r}.
\end{equation}
By (\ref{eq:22}), by Condition \ref{cond:condition}, and by the choice
$r\notin R$ in the algorithm's iterative step (3), the sequence $y_{r}^{k}$ is
well-defined for all integers $k$. Since $y_{r}^{k+1}$ is the projection of
$y_{r}^{k},$ for each integer $k\geq0$, we have%
\begin{equation}
\Vert z-y_{r}^{k}\Vert^{2}\geq\Vert z-y_{r}^{k+1}\Vert^{2}+\Vert y_{r}%
^{k}-y_{r}^{k+1}\Vert^{2}.\label{eq:eq (26)}%
\end{equation}
Consult, e.g., Theorem 1.2.4 and Lemma 1.2.5(c) in \cite{Cegielski-book}.
Using Equation (\ref{eq:eq (26)}) it follows that for each natural number
$q>2$,%
\begin{equation}
\Vert z-y_{r}^{1}\Vert^{2}\geq\sum_{k=1}^{q-1}(\Vert z-y_{r}^{k}\Vert
^{2}-\Vert z-y_{r}^{k+1}\Vert^{2})\geq\sum_{k=1}^{q-1}\Vert y_{r}^{k}%
-y_{r}^{k+1}\Vert^{2}%
\end{equation}
which, in turn, implies that%
\begin{equation}
\Vert z-y_{r}^{1}\Vert^{2}\geq\sum_{k=1}^{\infty}\Vert y_{r}^{k}-y_{r}%
^{k+1}\Vert^{2}%
\end{equation}
so that, as claimed in (\ref{eq:lim-0}),%
\begin{equation}
\lim_{k\rightarrow\infty}\Vert y_{r}^{k}-y_{r}^{k+1}\Vert=0.
\end{equation}
Let $\varepsilon>0$. There is a natural number $p$ such that for each integer
$k\geq pm$,%
\begin{equation}
\Vert y_{r}^{k}-y_{r}^{k+1}\Vert<\varepsilon/m.
\end{equation}
Therefore, for each pair of integers $k_{1},k_{2}\geq pm$, satisfying
$|k_{1}-k_{2}|\leq m$, we have%
\begin{equation}
\Vert y_{r}^{k_{1}}-y_{r}^{k_{2}}\Vert<\varepsilon.
\end{equation}
Thus, for each integer $g\geq p$ and each $i\in\{1,2,\dots,m\}$,%
\begin{equation}
dist(y_{r}^{gm+i},C_{s})<\varepsilon,\text{\ for all }s=1,2,\dots,m.
\end{equation}
This completes the proof of the theorem.
\end{proof}

Note that in our theorem we do not require the bounded regularity property
\cite[Definition 5.1]{bb96} to hold. If one adds this assumption then the
convergence of the sequence $\{y_{r}^{k}\}_{k=1}^{\infty}$ to a point of $C$
could be shown in a quite standard manner as follows.\textbf{ }Fix $z\in C\cap
C_{1,r}$. Then our sequence of iterates is Fejér monotone with respect to $z$,
thus, it is bounded. Then, because of (\ref{eq:(24)}) and bounded regularity,
for any $\varepsilon>0$ the sequence is in an $\varepsilon$-neighborhood of
some point from $C\cap C_{1,r}$ and, therefore, it converges to a point of
$C\cap C_{1,r}$.

\begin{proposition}
\label{prop:prop2.1}Let $\Omega=\cup_{s=1}^{m_{\Omega}}\Omega_{s}$ be a UCS
set. Assume that $D\subset\mathcal{H}$ is a nonempty convex set and that for
each $x\in D$, $P_{\Omega}(x)$ is a singleton. Then there exists an $1\leq
s\leq m_{\Omega}$ such that%
\begin{equation}
P_{\Omega}(x)\in\Omega_{s},\;\text{for all \ }x\in D.
\end{equation}

\end{proposition}

\begin{proof}
Assume to the contrary that there exist $x,y\in D$ and $s\in\{1,2,\dots
,m_{\Omega}\}$ such that%
\begin{equation}
P_{\Omega}(x)\in\Omega_{s},\;P_{\Omega}(y)\notin\Omega_{s}.\label{eq:2.2}%
\end{equation}
Define%
\begin{equation}
E:=\{\alpha\in\lbrack0,1]\mid\;P_{\Omega}(\beta x+(1-\beta)y)\in\Omega
_{s}\text{ \ for all \ }\beta\in\lbrack\alpha,1]\}.
\end{equation}
Clearly, $1\in E$. Define $\gamma:=\inf(E).$ Since the set $\Omega_{s}$ is
closed we have $\gamma\in E.$ By (\ref{eq:2.2}), $0\not \in E.$ There exists a
sequence $\{w_{i}\}_{i=1}^{\infty}\subset(0,\gamma)$ such that%
\begin{equation}
\lim_{i\rightarrow\infty}w_{i}=\gamma\label{eq:2.3}%
\end{equation}
and%
\begin{equation}
P_{\Omega}(w_{i}x+(1-w_{i})y)\notin\Omega_{s}.
\end{equation}
For each integer $i\geq1$,%
\begin{equation}
P_{\Omega}(w_{i}x+(1-w_{i})y)\in\cup\{P_{\Omega_{\ell}}\mid\;\ell
\in\{1,2,\dots,m_{\Omega}\}\setminus\{s\}\},
\end{equation}
which is a closed set, thus, in view of (\ref{eq:2.3}),%
\begin{equation}
P_{\Omega}(\gamma x+(1-\gamma)y)\in\cup\{P_{\Omega_{\ell}}\mid\;\ell
\in\{1,2,\dots,m_{\Omega}\}\setminus\{s\}\},
\end{equation}
which is a contraction that proves the proposition.
\end{proof}

\begin{example}
\label{example:example}Assume that ${\mathcal{H}}$ is a two-dimensional
Euclidean space, and define
\begin{equation}
C_{1,1}:=B((0,1),1),\;C_{1,2}:=C_{1,1}+(100,1),\;C_{1,3}:=C_{1,1}+(200,1),\;
\end{equation}%
\begin{equation}
C_{1,4}:=C_{1,1}+(-100,1),\;C_{1}:=\cup_{i=1}^{4}C_{1,i},
\end{equation}
and%
\begin{equation}
C_{2,1}:=B((0,-1),1),\;C_{2,2}:=C_{2,1}+(100,-1),\;C_{2}:=C_{2,1}\cup C_{2,2}.
\end{equation}
In this example, Condition \ref{cond:condition} holds and $R=\{1,2\}$ so that
if $r=1$ we solve a convex feasibility-seeking problem with two sets
$C_{1,1},C_{2,1}$ which has a unique solution $(0,0),$ and if $j=2$ we deal
with an inconsistent convex feasibility-seeking problem with sets
$C_{1,2},C_{2,2}$ which does not have a solution.
\end{example}

\section{Conclusion\label{sect:conclusions}}

We proposed an iterative process for the solution of the feasibility-seeking
problem for unions of convex sets (UCS) sets. This problem was recently
discovered to serve as a modeling approach in fields of applications and is
part of the ongoing recent research efforts to handle non-convexity in
feasibility-seeking problems.

Our convergence analysis relies on a technical condition (Condition
\ref{cond:condition}) for whose existence we present a sufficient condition
(Proposition \ref{prop:prop2.1}). We give an example (Example
\ref{example:example}) that shows that Condition \ref{cond:condition} is not a
vacuous condition. However, to identify and/or characterize classes of
problems that satisfy the condition is another question which we are unable to
answer at this point.\bigskip

\textbf{Acknowledgments.} The work of Y.C. is supported by the ISF-NSFC joint
research plan Grant Number 2874/19, by U.S. National Institutes of Health
Grant Number R01CA266467 and by the Cooperation Program in Cancer Research of
the German Cancer Research Center (DKFZ) and Israel's Ministry of Innovation,
Science and Technology (MOST).


\begin{thebibliography}{99}                                                                                               %


\bibitem {bb96}H.H. Bauschke and J.M. Borwein, \textquotedblleft On projection
algorithms for solving convex feasibility problems\textquotedblright,
\textit{SIAM Review}, \textbf{38}: 367--426, (1996). https://doi.org/10.1137/S0036144593251710.

\bibitem {bauschke-noll-2014}H.H. Bauschke and D. Noll, \textquotedblleft On
the local convergence of the Douglas--Rachford algorithm\textquotedblright,
Archiv der Mathematik, \textbf{102}: 589--600 (2014). https://doi.org/10.1007/s00013-014-0652-2

\bibitem {bauschke-2014}H.H. Bauschke, H.M. Phan and X. Wang,
\textquotedblleft The method of alternating relaxed projections for two
nonconvex sets\textquotedblright, \textit{Vietnam Journal of Mathematics,}
\textbf{42}: 421--450, (2014). https://doi.org/10.1007/s10013-013-0049-8.

\bibitem {Cegielski-book}A. Cegielski, \textit{Iterative Methods for Fixed
Point Problems in Hilbert Spaces}, Springer-Verlag Berlin, Heidelberg,
2012. https://link.springer.com/book/10.1007/978-3-642-30901-4.

\bibitem {censor-cegielski-review}Y. Censor and A. Cegielski,
\textquotedblleft Projection methods: an annotated bibliography of books and
reviews\textquotedblright, \textit{Optimization}, \textbf{64}: 2343--2358,
(2015). https://doi.org/10.1080/02331934.2014.957701.

\bibitem {chertien-bondon-1996}S. Chrétien and P. Bondon, \textquotedblleft
Cyclic projection methods on a class of nonconvex sets\textquotedblright,
\textit{Numerical Functional Analysis and Optimization}, \textbf{17}: 37--56,
(1996). http://dx.doi.org/10.1080/01630569608816681.

\bibitem {cheritien-bondon-2020}S. Chrétien and P. Bondon, \textquotedblleft
Projection Methods for Uniformly Convex Expandable Sets\textquotedblright,
\textit{Mathematics} \textbf{8}: 1108, (2020). https://doi.org/10.3390/math8071108.

\bibitem {corvellec-2004}J.-N. Corvellec and S.D. Flåm, Non-convex feasibility
problems and proximal point methods, \textit{Optimization Methods and
Software}, \textbf{19}: 3--14, (2004). https://doi.org/10.1080/10556780410001654223.

\bibitem {dao-2019}M.N. Dao and M.K. Tam, \textquotedblleft Union averaged
operators with applications to proximal algorithms for min-convex
functions\textquotedblright, \textit{Journal of Optimization Theory and
Applications}, \textbf{181}: 61--94, (2019).\newline https://doi.org/10.1007/s10957-018-1443-x.

\bibitem {dizon-2022}N.D. Dizon, J.A. Hogan and S.B. Lindstrom,
\textquotedblleft Circumcentering reflection methods for nonconvex feasibility
problems\textquotedblright, \textit{Set-Valued and Variational Analysis},
\textbf{30}: 943--973 (2022). https://doi.org/10.1007/s11228-021-00626-9.

\bibitem {Firouzeh-2022}F.F. Firouzeh, J.W. Chinneck and S. Rajan,
\textquotedblleft Faster Maximum Feasible Subsystem solutions for dense
constraint matrices\textquotedblright, \textit{Computers \& Operations
Research}, \textbf{139}, (2022) 105633. \newline https://www.sciencedirect.com/science/article/pii/S0305054821003439.

\bibitem {hesse-luke}R. Hesse and D.R. Luke. \textquotedblleft Nonconvex
notions of regularity and convergence of fundamental algorithms for
feasibility problems\textquotedblright, \textit{SIAM Journal on Optimization,}
\textbf{23}: 2397--2419 (2013).

\bibitem {pong-li-2016}G. Li and T.K. Pong, \textquotedblleft
Douglas--Rachford splitting for nonconvex optimization with application to
nonconvex feasibility problems\textquotedblright, \textit{Mathematical
Programming}, \textbf{159}: 371--401 (2016).\newline https://doi.org/10.1007/s10107-015-0963-5.

\bibitem {tam-2018}M.K. Tam, \textquotedblleft Algorithms based on unions of
nonexpansive maps\textquotedblright, \textit{Optimization Letters},
\textbf{12}, 1019--1027 (2018). https://doi.org/10.1007/s11590-018-1249-7.

\bibitem {Shan-Yu-2023}S. Yu, Y. Censor, M. Jiang and G. Luo,
\textquotedblleft Per-RMAP: Feasibility-Seeking and Superiorization Methods
for Floorplanning with I/O Assignment\textquotedblright, in:
\textit{Proceedings of the 2023 International Symposium of Electronics Design
Automation (ISEDA) ISEDA-2023}, Nanjing, China, May 8-11, 2023, pp. 286--291,
(2023).\newline https://ieeexplore.ieee.org/document/10218694.

\bibitem {zaslavski-2023}A.J. Zaslavski, \textquotedblleft Approximate
solutions of a fixed-point problem with an algorithm based on unions of
nonexpansive mappings\textquotedblright, \textit{Mathematics} \textbf{11}:
1534, (2023). https://doi.org/10.3390/math11061534.
\end{thebibliography}
\end{document}